\documentclass[11pt]{amsart}
\usepackage{colortbl}
\newtheorem{theo}{THEOREM}[section]
\newtheorem{lemma}[theo]{Lemma}
\newtheorem{cor}[theo]{Corollary}
\newtheorem{prop}[theo]{Proposition}
\newtheorem{dfntn}[theo]{Definition}
\newtheorem{rem}[theo]{Remark}
\newtheorem{claim}[theo]{Claim}

\newtheorem{ass}[theo]{Assumptions}

\newtheorem{ex}[theo]{Example}
\newtheorem{facts}[theo]{FACTS}
\newcommand{\brref}[1]{(\ref{#1})}

\newcommand{\Pin}[1]{{\mathbb P}^{#1}}
\newcommand{\restrict}[2]{{#1}_{\mid _{#2}}}
\newcommand{\lra}{\longrightarrow}

\newcommand{\Projcal}[1]{\mathbb{ P}({\mathcal #1})}

\newcommand{\oofp}[2]{{\mathcal O}_{\mathbb{ P}^{#1}}({#2})}

\newcommand{\scrollcal}[1]{(\Projcal{#1},\tautcal{#1})}
\newcommand {\xel} {(X, L)}

\newcommand{\tautcal}[1]{{\mathcal O}_{\mathbb{P}({\mathcal#1})}(1)}

\newcommand{\num}{\equiv}

\newcommand{\Pp}{\mathbb P}
\newcommand{\Oc}{\mathcal O}
\newenvironment{rem*}{\begin{rem}\em}{\end{rem}}
\newenvironment{ex*}{\begin{ex}\em}{\end{ex}}
\newenvironment{claim*}{\begin{claim}\em}{\end{claim}}
\newenvironment{facts*}{\begin{facts}\em}{\end{facts}}

\newcommand{\FF}{\mathbb{F}}

\usepackage{lineno}

\allowdisplaybreaks[1]

\title[Hilbert scheme of some threefold scrolls over ${\FF}_1$]{Hilbert scheme of some threefold scrolls over the Hirzebruch surface ${\FF}_1$}

\author{Gian Mario Besana}
\address{Gian Mario Besana\\
College of Computing and Digital Media\\ DePaul University\\
1 E. Jackson \\
Chicago IL 60604\\USA}
\email{gbesana@depaul.edu}

\author{Maria Lucia Fania}
\address{Maria Lucia Fania\\ Dipartimento di Ingegneria e Scienze dell'Informazione e Matematica\\
Universit\`{a} degli Studi di L'Aquila\\
Via Vetoio Loc. Coppito\\67100 L'Aquila\\Italy}
\email{fania@univaq.it}

\author{Flaminio Flamini}
\address{Flaminio Flamini\\Dipartimento di Matematica\\ Universit\`a degli Studi di Roma
Tor Vergata \\ Viale della Ricerca Scientifica, 1 - 00133 Roma\\Italy}
\email{flamini@mat.uniroma2.it}

\keywords{Hilbert scheme, Special threefolds, Vector  bundles, Ruled varieties}
\subjclass[2000]{Primary 14J30,14M07,14N25; Secondary 14N30}

\thanks{The authors wish to warmly thank C. Ciliberto and E. Sernesi for useful discussions and L. Ein for fundamental remarks concerning the proof of Theorem \ref{thm:van}.}

\begin{document}


\begin{abstract} Hilbert schemes of suitable smooth, projective manifolds of low degree which are $3$-fold scrolls over the Hirzebruch  surface $\FF_1$ are studied. An irreducible component of the Hilbert scheme parametrizing such varieties is shown to be generically smooth of the expected dimension  and  the general point of such a component is  described.
\end{abstract}


\maketitle

\section{Introduction}
The Hilbert scheme of complex projective manifolds with given Hilbert polynomial has interested several authors over the years. Ellingsrud, \cite{e}, dealt with arithmetically Cohen-Macaulay varieties of codimension two, while the Hilbert scheme of a special class of $3$-folds in $\Pin{5}$ was studied by Fania and Mezzetti,
\cite{fa-me}. General results in codimension two are also due to
M.C. Chang, \cite{chang1}, \cite{chang2}.

In codimension three, Kleppe and Mir\'o-Roig, \cite{kl-mr}, considered the case  of
arithmetically Gorenstein closed subschemes,  while Kleppe, Migliore, Mir\'o-Roig, Nagel and
Peterson, \cite{kmmnp}, dealt with good determinantal subschemes.

In the case of higher codimension, the Hilbert scheme of special classes of varieties was also studied.  For instance, the case of a Palatini scroll in $\Pin{n}$, with $n$ odd,  was considered by Faenzi and Fania,  \cite{fae-fan}.  The two dimensional version of the Palatini scroll is particularly well-studied, in the framework of surface scrolls in $\Pin{n}$ which are non-special.
In particular we mention the results on the
Hilbert schemes of non-special scrolls due to Calabri, Ciliberto, Flamini, and Miranda,
\cite{calabri-ciliberto-flamini-miranda:non-special}.

Previously  the first two authors considered several classes of $3$-folds in  $\Pin{n},$ $n \ge 6,$ and  computed the dimension of an irreducible component of their Hilbert scheme, \cite{be-fa} .
Considering the  existing classification of complex projective manifolds of low degree, the Hilbert scheme of classes of $3$-folds which are scrolls over the Hirzebruch surface
 $\FF_1$ emerges as a natural possible object of study. Further motivation to address this issue comes from the fact that Alzati and Besana, \cite{al-be}, established the existence of $3$-fold scrolls over  $\FF_1,$ of low degree, as a byproduct of a very ampleness criterion  for a particular class of vector bundles on  $\FF_1.$  In this work the Hilbert scheme of such a class of $3$-folds is therefore considered. An irreducible component of the Hilbert scheme parametrizing such varieties is shown to be generically smooth of the expected dimension (cf. Proposition \ref{Hilbertscheme of scrollover F1}) and  the general point of such a component is described (cf. Theorem \ref{thm:genpo}).

The paper is organized as follows. In \S\,\ref{notation}  notation and terminology used in the paper are fixed once and for all.
In \S\,\ref{S:vbF1}, based on previous results in \cite{al-be}, the vector bundles on $\FF_1$ that are of interest in this work are described (cf. Assumptions \ref{ass:AB});  cohomological properties of such vector bundles (cf. Corollary \ref{cor:comEAB} and Lemma \ref{lem:autE}) as well as of the family of extension classes parametrizing them (cf. Lemma \ref{lem:ext1}) are then established. The main result of the section is Theorem  \ref{thm:van} where, under suitable numerical assumptions, it is shown that the general vector bundle $\mathcal E$ in the extension class parameter space is {\em non-special}, i.e. $h^1(\FF_1, \mathcal E) = 0.$ Section \ref{reinterpretation} offers other interpretations of the non-speciality of the vector bundles under consideration  in terms of suitable coboundary maps and cup products of divisors on $\FF_1$  as well as of projective geometry of suitable Segre varieties (cf. Corollary \ref{cor:van} and what follows).
In \S\,\ref{S:scrollsF1} the focus is on Hilbert schemes parametrizing families of $3$-dimensional scrolls over $\FF_1$ defined by the vector bundles studied in the previous sections. Indeed, as in \cite{al-be}, under assumptions giving necessary conditions for a general $\mathcal E$ to be very ample on $\FF_1$ (cf. Assumptions \ref{ass:AB2}), the ruled projective variety $\Pp(\mathcal E)$ embedded via the tautological linear system is studied; as $\mathcal E$ varies in the parameter space of extensions, the associated ruled varieties are shown to fill-up an irreducible component of the Hilbert scheme parametrizing such varieties; such a component is then shown to be generically smooth and of the expected dimension (cf. Proposition \ref{Hilbertscheme of scrollover F1}) and a description of  the general point of such a component is offered (cf. Theorem \ref{thm:genpo}).
In \S\,\ref{Examples},  key examples of low-degree scrolls on $\FF_1$ are discussed, and examples of scrolls over a quadric surface, which were previously studied  in \cite{be-fa}, are reinterpreted.


\section{Notation and Preliminaries}
\label{notation}
The following notation will be used throughout this work.
  \begin{enumerate}
\item [ ] $X$ is a smooth, irreducible, projective variety of dimension $3$ (or simply a $3$-fold);
\item[ ]$\chi(\mathcal F) = \sum(-1)^ i h^i(\mathcal F)$, the Euler characteristic of $\mathcal F$, where $\mathcal F$ is any vector bundle of rank $r \geq 1$ on $X$;
\item[ ]  $\restrict{\mathcal F}{Y}$ the restriction of $\mathcal F$ to a subvariety $Y;$
\item[ ]  $K_X$ the canonical bundle of $X.$ When the context is clear, $X$ may be dropped, so $K_X = K$;
\item[ ] $c_i = c_i(X)$,  the $i^{th}$ Chern class
of $X$;
\item[ ] $d = \deg{X} = L^3$, the degree of $X$ in the embedding
given by a very-ample line bundle $L$;
\item[] $g = g(X),$ the sectional genus of $\xel$ defined by
$2g-2=(K+2L)L^2;$
\item[ ] if $S$ is a smooth surface, then $q(S) = h^1 ({\mathcal O}_{S})$ denotes the irregularity of $S$, whereas $p_g(S) = h^0 (K_S)$, denotes the geometric genus of $S$;
\item[] $\equiv$ will denote the numerical equivalence of divisors on a smooth surface $S$;
\end{enumerate}

\begin{dfntn}\label{specialvar} A pair $(X, L)$, where
$L$ is an ample line bundle on a $3$-fold $X,$ is
a {\it scroll} over a normal variety $Y$ if there exist an ample line
bundle $M$ on $Y$ and a surjective morphism  $\varphi: X \to Y$ with
connected fibers
such that $K_X + (4 - \dim Y) L = \varphi^*(M).$
\end{dfntn}

In particular, if $Y$ is smooth and $\xel$ is a scroll over $Y$,  then (see \cite[Prop. 14.1.3]{BESO})
$X \cong \Projcal{E} $, where ${\mathcal E}= \varphi_{*}(L)$ and
$L$ is the tautological  line bundle on $\Projcal{E}.$ Moreover, if $S \in |L|$ is a smooth divisor,
then (see e.g. \cite[Thm. 11.1.2]{BESO}) $S$ is the blow
up of $Y$ at $c_2(\mathcal{E})$ points; therefore $\chi({\mathcal O}_{Y}) = \chi({\mathcal O}_{S})$ and
\begin{equation}\label{eq:d}
d : = L^3 = c_1^2(\mathcal{E})-c_2(\mathcal{E}).
\end{equation}

Throughout this work the scroll's base $Y $ will be the Hirzebruch surface $\FF_1$ defined as
$\FF_1 = \Pp(\Oc_{\Pp^1} \oplus\Oc_{\Pp^1}(-1))$.
Let $$\pi : \FF_1 \to \Pp^1$$be the natural projection onto the base.
It is well-known that $${\rm Num}(\FF_1) = \mathbb{Z}[C_0] \oplus \mathbb{Z}[f],$$where
\noindent

$\bullet$ $C_0$ denotes the unique section of $\FF_1$
corresponding to the quotient bundle morphism $\Oc_{\Pp^1} \oplus\Oc_{\Pp^1}(-1) \to\!\!\!\to \Oc_{\Pp^1}(-1)$ on $\Pp^1$, and
\noindent

$\bullet$ $f = \pi^*(p)$, for any $p \in \Pp^1$.
\noindent

In particular $$C_0^2 = - 1, \; f^2 = 0, \; C_0f = 1.$$

\section{Some rank-two vector bundles over $\FF_1$} \label{S:vbF1}

Let $\mathcal{E}$ be a rank-two vector bundle over $\FF_1$ and let $c_i(\mathcal{E})$ be the $i^{th}$-Chern class of $\mathcal E$, $1 \leq i \leq2$.
Then $c_1( \mathcal{E}) \num a C_0 + b f$, for some $ a, b \in \mathbb Z$, and $c_2(\mathcal{E}) \in \mathbb Z$.

From now on (cf. \S\,\ref{S:scrollsF1} and Proposition \ref{prop:AB} for motivation), we will use:

\begin{ass}\label{ass:AB} Let $\mathcal E$ be a rank-two vector bundle over
$\FF_1$ such that $$c_1(\mathcal E) \num 3 C_0 + b f, \;\; c_2(\mathcal E) = k, \;\; {\rm with} \;\; k \geq b \geq 4.$$ Moreover, we assume there exists an exact sequence
\begin{equation}\label{eq:al-be}
0 \to A \to \mathcal E \to B \to 0,
\end{equation}where $A$ and $B$ are line bundles on $\FF_1$ such that
\begin{equation}\label{eq:al-be3}
A \num 2 C_0 + (2b-k-2) f \;\; {\rm and} \;\; B \num C_0 + (k - b + 2) f.
\end{equation}

\end{ass} In particular, $c_1(\mathcal E) = A + B \;\; {\rm and} \;\; k = c_2(\mathcal E) = AB$.

Note that the exact sequence \eqref{eq:al-be} gives important preliminary information on the cohomology of $\mathcal E$, $A$ and $B$. Indeed, one has

\begin{prop}\label{prop:comEAB} With hypotheses as in Assumptions \ref{ass:AB},
\begin{equation}\label{eq:comEAB}
h^j(\FF_1, \mathcal E) = h^j(\FF_1, A) = 0, \;\; j \geq 2, \;\; {\rm and} \;\; h^i(\FF_1, B) = 0, \;\; i \geq 1.
\end{equation}
\end{prop}

\begin{proof} It is clear that
$$h^j(\FF_1, \mathcal E) = h^j(\FF_1, A) = h^j(\FF_1, B) = 0, \; j \geq 3,$$for dimension reasons.
Observe also that $$h^2(\FF_1, A) = h^2(\FF_1, B) = 0.$$Indeed, by Serre's duality, we have
$h^2(A ) = h^0(K_{\FF_1} - A)$ and $h^2(B ) = h^0(K_{\FF_1} - B)$. Since $K_{\FF_1} \num
- 2 C_0 - 3 f$  then, from \eqref{eq:al-be3},
$$K_{\FF_1} - A \num - 4C_0 - (2b-k+1) f, \;\;\; K_{\FF_1} - B \num - 3C_0 - (k-b+5) f$$which cannot be effective, since they both negatively intersect the irreducible divisor $f$. In particular, this also implies
$$h^2(\FF_1, \mathcal E) = 0.$$

We claim that, under Assumptions \ref{ass:AB}, we also have $$h^1(\FF_1, B) = 0.$$ Indeed,
$B \num K_{\FF_1} + B'$, where $ B' \equiv 3 C_0 + (k -b +5) f$. Since $k \geq b$, by \cite[Cor. 2.18, p. 380]{H} we deduce that $B'$ is ample. Therefore, from the Kodaira vanishing theorem,  we are done.
\end{proof}

From \eqref{eq:al-be} and Proposition \ref{prop:comEAB}, we have:
\begin{equation}\label{eq:comseq}
0 \to H^0(A) \to H^0(\mathcal E) \to H^0(B) \stackrel{\partial}{\longrightarrow} H^1(A) \to H^1(\mathcal E) \to 0,
\end{equation}where $\partial$ is the {\em coboundary map} determined by \eqref{eq:al-be}; in particular,
\begin{equation}\label{eq:h1Eh1A}
h^1(\mathcal E) \leq h^1(A).
\end{equation}

\begin{cor}\label{cor:comEAB} With hypotheses as in Assumptions \ref{ass:AB}, one has
\begin{equation}\label{eq:h0A}
h^0(A) = 6b-3k-6 + h^1(A),
\end{equation}
\begin{equation}\label{eq:h0E}
h^0(\mathcal E) = 4b-k-1 + h^1(\mathcal E).
\end{equation}
\begin{equation}\label{eq:h0B}
h^0(B) = 2k-2b+5,
\end{equation}
\end{cor}
\begin{proof} From  Proposition \ref{prop:comEAB}, we have
\begin{equation}\label{eq:comseq2}
\chi(A) = h^0(A) - h^1(A), \;\; \chi(B) =h^0(B), \;\; \chi(\mathcal E) = h^0(\mathcal E) - h^1(\mathcal E).
\end{equation}From the Riemann-Roch formula, we have
$$\chi(A) = \frac{1}{2}A (A-K_{\FF_1}) + 1 = \frac{1}{2}A (A') + 1 = $$
$$\frac{1}{2} \left(2C_0 + (2b-k-2) f \right) \left(4C_0 + (2b-k+1)f\right) + 1 = 6b-3k-6,$$whereas
$$\chi(B) = h^0(B) =  \frac{1}{2} B (B-K_{\FF_1}) + 1 = \frac{1}{2}B (B') + 1 = $$
$$\frac{1}{2} \left(C_0 + (k-b+2) f \right) \left(3C_0 + (k-b+5)f\right) + 1 = 2k-2b+5.$$Since $\chi(\mathcal E) =
\chi(A) + \chi(B)$, the remaining statements follow from \eqref{eq:comseq} and \eqref{eq:comseq2}.
\end{proof}

\subsection{Vector bundles in $Ext^1(B,A)$}\label{S:vb} This subsection is devoted to an analysis of
vector bundles fitting in the exact sequence \eqref{eq:al-be}.

We need the following:
\begin{lemma}\label{lem:ext1} With hypotheses as in Assumptions \ref{ass:AB}, one has
\begin{equation}\label{eq:dimExt1}
\dim(Ext^1(B,A)) = \left\{
\begin{array}{ccc}
0 & & b \leq k <  \frac{3b-3}{2}\\
4k-6b+7 & & k \geq \frac{3b-3}{2}.
\end{array}
\right.
\end{equation}
\end{lemma}

\begin{proof} By standard facts, $Ext^1(B,A) \cong H^1(A-B)$. From Assumptions \ref{ass:AB},

$$A - B \equiv 2C_0 + (2b-k-2) f - C_0 - (k-b+2)f = C_0 + (3b-2k-4)f.$$From Leray's isomorphism,
$$h^1(\FF_1, C_0 + (3b-2k-4)f) = h^1\left(\Pp^1, (\Oc_{\Pp^1} \oplus \Oc_{\Pp^1} (-1)) \otimes \Oc_{\Pp^1} (3b-2k-4)\right).$$The latter equals $h^1(\Oc_{\Pp^1} (3b-2k-4)) + h^1(\Oc_{\Pp^1} (3b-2k-5))$ which, by Serre's duality, coincides with
$h^0(\Oc_{\Pp^1} (2k-3b+2)) + h^0( \Oc_{\Pp^1} (2k-3b+3))$. The statement  immediately follows.
\end{proof}

In particular, we have

\begin{cor}\label{cor:dimExt1} With hypotheses as in Assumptions \ref{ass:AB},
\begin{itemize}
\item[(i)] for $b \leq k < \frac{3b-3}{2}$, $\mathcal E \in Ext^1(B,A)$ splits, i.e. $\mathcal E = A \oplus B$;
\item[(ii)] for $k \geq  \frac{3b-3}{2}$, the general vector bundle $\mathcal E \in Ext^1(B,A)$ is indecomposable.
\end{itemize}
\end{cor}

In \S\,\ref{S:genpoint}, we shall also need to know $\dim(Aut(\mathcal E)) = h^0(\mathcal E \otimes \mathcal E^{\vee})$.

\begin{lemma}\label{lem:autE} With hypotheses as in Assumptions \ref{ass:AB},
\begin{equation}\label{eq:autE}
h^0(\mathcal E \otimes \mathcal E^{\vee}) = \left\{
\begin{array}{ccl}
6b-4k -5 & & b \leq k <   \frac{3b-3}{2}\\
1 & & k \geq \frac{3b-3}{2} \;\; {\rm for} \;\; \mathcal E \;\; {\rm indecomposable}.
\end{array}
\right.
\end{equation}
\end{lemma}
\begin{proof} (i) According to Corollary \ref{cor:dimExt1}, for $k < \frac{3b-3}{2}$, $\mathcal E = A \oplus B$.
Therefore
$$\mathcal E \otimes \mathcal E^{\vee} \cong \Oc^{\oplus 2} \oplus (A-B) \oplus (B-A).$$From Assumptions \ref{ass:AB},
$(B-A) \equiv - C_0 + (2k-3b+4) f$ so it is not effective, since it negatively intersects the irreducible, moving curve $f$.

On the contrary, $(A-B) \equiv C_0 + (3b-2k-4) f$. As in the proof of Lemma \ref{lem:ext1},
$h^0(\FF_1, C_0 + (3b-2k-4) f) = h^0(\Oc_{\Pp^1} (3b-2k-4)) + h^0(\Oc_{\Pp^1} (3b-2k-5))$.
\noindent

$\bullet$ Observe that, for $k <  \frac{3b-5}{2}$, $\Oc_{\Pp^1} (3b-2k-4)$ and $\Oc_{\Pp^1} (3b-2k-5)$ are both effective. So
$h^0(\Oc_{\Pp^1} (3b-2k-4)) + h^0(\Oc_{\Pp^1} (3b-2k-5)) = 6b-4k-7$; taking into account also
$h^0(\Oc^{\oplus 2})$, we conclude in this case.
\noindent

$\bullet$ For $k = \frac{3b-4}{2}$, $\Oc_{\Pp^1} (3b-2k-4) \cong \Oc_{\Pp^1}$ whereas $\Oc_{\Pp^1} (3b-2k-5) \cong \Oc_{\Pp^1}(-1)$; therefore,
$h^0(\FF_1, C_0 + (3b-2k-4) f) = 1$, so $h^0(\mathcal E \otimes \mathcal E^{\vee}) = 3 = 6b-4k -5$.
\vskip 5pt

\noindent

(ii) Assume now $k \geq \frac{3b-3}{2}$. From Corollary \ref{cor:dimExt1}, the general vector bundle $\mathcal E \in Ext^1(B,A)$ is indecomposable. Using the fact that $\mathcal E $ is of rank two and fits in the exact sequence \eqref{eq:al-be}, we have
$$\mathcal E^{\vee} \cong \mathcal E \otimes \Oc(-A-B),$$since $c_1(\mathcal E) = A+B$. Therefore, tensoring \eqref{eq:al-be} by $\mathcal E^{\vee}$ gives
\begin{equation}\label{eq:al-betwist}
0 \to \mathcal E (-B) \to \mathcal E \otimes \mathcal E^{\vee} \to \mathcal E (-A) \to 0.
\end{equation}One has $h^0(\mathcal E \otimes \mathcal E^{\vee}) \geq 1$, since scalar multiplication
always determines an automorphism of $\mathcal E$. We want to show that equality holds.

To do this, we want to compute both $h^0(\mathcal E (-B))$ and $h^0(\mathcal E (-A))$.

To compute the first, tensor \eqref{eq:al-be} by $\Oc(-B)$, and  get
\begin{equation}\label{eq:al-betwist2}
0 \to (A-B) \to \mathcal E (-B) \to \Oc \to 0.
\end{equation} Since $k \geq \frac{3b-3}{2}$, the same computations used in part (i) of the proof
show that in this case $A-B$ is not effective. Furthermore, observe that the coboundary map $$H^0(\Oc) \stackrel{\partial}{\longrightarrow} H^1(A-B),$$arising from \eqref{eq:al-betwist2}, has to be injective since it corresponds to the choice of the extension class
$e \in Ext^1(B,A)$ associated to $\mathcal E$; in other words, also $h^0( \mathcal E (-B)) = 0$.

To compute  $h^0(\mathcal E (-A))$, tensor \eqref{eq:al-be} by $\Oc(-A)$ and get
\begin{equation}\label{eq:al-betwist3}
0 \to \Oc \to \mathcal E (-A) \to (B-A) \to 0.
\end{equation}As in part (i) of the proof, $B-A$ is not effective. Therefore, from \eqref{eq:al-betwist3}, we get
$h^0(\mathcal E (-A)) = 1$.

From \eqref{eq:al-betwist}, we deduce also $h^0(\mathcal E \otimes \mathcal E^{\vee}) \leq 1$, proving
\eqref{eq:autE} in this case.
\end{proof}

\subsection{Computation of $h^1(\mathcal E)$}\label{S:h1E}
The main result of this subsection (cf. Theorem \ref{thm:van})
is about  the non-speciality of the (general) vector bundle $\mathcal E$ as in \eqref{eq:al-be}, under suitable numerical assumptions.

To start with, recall that, from Proposition \ref{prop:comEAB}, $h^1(B) = 0$, for any $k \geq b$.
We now compute $h^1(A)$.

\begin{lemma}\label{lem:h1A} With hypotheses as in Assumptions \ref{ass:AB}, one has
\begin{equation}\label{eq:h1A}
h^1(A) = \left\{
\begin{array}{ccc}
0 & & b \leq k \leq 2b-3 \\
1 & & k = 2b-2 \\
3k-6b+6 & & k \geq 2b-1.
\end{array}
\right.
\end{equation}
\end{lemma}
\begin{proof} (i) Consider the case $ k \leq 2b-3$. To prove $h^1(A) = 0$, we can write $A \num K_{\FF_1} + A'$, where $A' \equiv 4 C_0 + (2b-k + 1) f$.
From \cite[Cor. 2.18, p. 380]{H}, $A'$ is ample if and only if $k \leq 2b-4$. From the Kodaira vanishing theorem, $h^1(\FF_1, A) = 0$. For $k = 2b-3$, $A' \equiv 4 C_0 + 4 f$ which is effective and such that $(A')^2 = 4$, therefore it is big and nef. From the Kawamata-Viehweg vanishing theorem, $h^1(\FF_1, A) = 0$ also in this case.
\vskip 5pt

\noindent

(ii) For  $k = 2b-2$, we have $A \equiv 2 C_0$, which is effective. Thus, from the exact sequences
$$0 \to \Oc_{\FF_1} (C_0) \to \Oc_{\FF_1}(2C_0) \to \Oc_{C_0}(2C_0) \cong \Oc_{\Pp^1}(-2) \to 0$$and
$$0 \to \Oc_{\FF_1} \to \Oc_{\FF_1} (C_0) \to \Oc_{C_0}(C_0) \cong \Oc_{\Pp^1}(-1) \to 0$$we immediately get
that $h^1(A) = h^1(\Oc_{\FF_1}(2C_0)) = h^1(\Oc_{\Pp^1}(-2)) = 1$.
\vskip 5pt

\noindent

(iii) For $k \geq 2b-1$, we claim that $A$ is not effective. Indeed, $AC_0 = 2b-k-4$ and since $k \geq 2b-1$, then $AC_0 \leq -3$; $C_0$ cannot be a fixed component of $|A|$, since $A-C_0 \equiv C_0 + (2b-k-2)f$ is not effective, as it is clear from $(A-C_0) C_0 = 2b-k-3 \leq -2$. Therefore $h^0(A) = 0$, so we conclude from \eqref{eq:h0A}.
\end{proof}

We can now prove the main result of this subsection.

\begin{theo}\label{thm:van} Let  $\mathcal E \in Ext^1(B,A)$ be as in Assumptions \ref{ass:AB}. If
\begin{itemize}
\item[(i)]  $b \leq k \leq 2b-3,$
\item [ ] or
\item[(ii)]  $2b-2 \leq k \leq 4b-1$ and  $\mathcal E \in Ext^1(B,A)$ general,
\end{itemize}
 then $h^1(\mathcal E) = 0.$
\end{theo}
\begin{proof} (i) For  $b \leq k \leq 2b-3$, the statement  follows directly from \eqref{eq:h1Eh1A} and
Lemma \ref{lem:h1A}.
\noindent

(ii) For $2b-2 \leq k \leq 4b-1$, consider the exact sequence \eqref{eq:al-be} and the natural morphism
$\pi : \FF_1 \to \Pp^1$. From assumptions \eqref{eq:al-be3}, applying $\pi_*$ to \eqref{eq:al-be} gives
the exact sequence of vector bundles on $\Pp^1$
\begin{equation}\label{eq:esp1}
{\small
\begin{array}{rlll}
0 \to & \pi_*(A) \cong Sym^2(\Oc_{\Pp^1} \oplus \Oc_{\Pp^1}(-1)) \otimes \Oc_{\Pp^1}(2b-k-2) & \to &  \pi_*(\mathcal E) \\
\to & \pi_*(B) \cong (\Oc_{\Pp^1} \oplus \Oc_{\Pp^1}(-1)) \otimes \Oc_{\Pp^1}(k-b+2) & \to & 0.
\end{array}
}
\end{equation}

\noindent

By standard computations on symmetric powers of vector bundles, \eqref{eq:esp1} gives
\begin{equation}\label{eq:esp2}
{\small
\begin{array}{rlll}
0 \to & \Oc_{\Pp^1} (2b-k-2) \oplus \Oc_{\Pp^1}(2b-k-3) \oplus  \Oc_{\Pp^1}(2b-k-4) & \to &  \pi_*(\mathcal E) \\
\to & \Oc_{\Pp^1} (k-b+2) \oplus \Oc_{\Pp^1}(k-b+1) & \to & 0.
\end{array}
}
\end{equation}Leray's isomorphisms give bijective correspondences between extensions classes as well as their cohomological behaviour: in other words, $$h^i(\FF_1, \mathcal E) \cong h^i(\Pp^1, \pi_*(\mathcal E)), \; i \geq 0,$$ as well as the cohomological class corresponding to the general extension $\mathcal E \in Ext^1(B, A)$ on $\FF_1$ is isomorphic
to the cohomological class corresponding to the general extension $\pi_*(\mathcal E) \in Ext^1(\Oc_{\Pp^1} (k-b+2) \oplus \Oc_{\Pp^1}(k-b+1),  \Oc_{\Pp^1} (2b-k-2) \oplus \Oc_{\Pp^1}(2b-k-3) \oplus  \Oc_{\Pp^1}(2b-k-4))$ on $\Pp^1$.

In particular, $\pi_*(\mathcal E)$ is a rank-five vector bundle on $\Pp^1$, with
\begin{equation}\label{eq:degfico}
\deg(\pi_*(\mathcal E)) = 4b-k-6.
\end{equation}Since we are on $\Pp^1$, $\pi_*(\mathcal E)$ is decomposable (cf. e.g. \cite[Thm. 2.1.1]{OSS}), i.e.
\begin{equation}\label{eq:degfico2}
\pi_*(\mathcal E) = \bigoplus_{i=1}^5 \Oc_{\Pp^1}(\alpha_i),
\end{equation}for some $\alpha_i \in \mathbb{Z}$, $1 \leq i \leq 5$, such that from \eqref{eq:degfico}
$$\Sigma_{i=1}^5 \alpha_i = 4b-k-6.$$ If $\mathcal E \in Ext^1(B,A)$ (equivalently $\pi_*(\mathcal E) \in Ext^1(\Oc_{\Pp^1} (k-b+2) \oplus \Oc_{\Pp^1}(k-b+1),  \Oc_{\Pp^1} (2b-k-2) \oplus \Oc_{\Pp^1}(2b-k-3) \oplus  \Oc_{\Pp^1}(2b-k-4))$) is general,  then
$(\alpha_1, \ldots, \alpha_5)$ must be as balanced as possible (cf. e.g. \cite[\S\;2]{OSS}).

Therefore, letting $4b-k-6 \equiv \epsilon \pmod{5},$ we have to consider five different cases according to the values of the integer $0 \leq \epsilon \leq 4$.

\begin{itemize}
\item[$\epsilon = 0$] In this case, $\pi_*(\mathcal E)$ general implies $\alpha_i= \frac{4b-k-6}{5}$, for any $1 \leq i \leq 5$; in other words $\pi_*(\mathcal E) = \Oc_{\Pp^1}^{\oplus 5}(\frac{4b-k-6}{5})$. Therefore, $h^1(\pi_*(\mathcal E)) = 5 \, h^1( \Oc_{\Pp^1}(\frac{4b-k-6}{5}))$ and the latter equals zero as soon as $k < 4b+4$, which is part of our numerical assumptions.
\item[$\epsilon = 1$] In this case, we can write $4b-k-6 = 5 h + 1$, where $h = \frac{4b-k-7}{5}$. The generality assumptions on $\pi_*(\mathcal E)$ implies that (up to a permutation of the integers $\alpha_i$)
$$\alpha_1 = \alpha_2 = \alpha_3 = \alpha_4 = h = \frac{4b-k-7}{5}$$and$$\alpha_5 = h+1 = \frac{4b-k-7}{5} + 1.$$Therefore, $$h^1(\pi_*(\mathcal E)) = 4 h^1 \left( \Oc_{\Pp^1}(\frac{4b-k-7}{5})\right) + h^1\left( \Oc_{\Pp^1}(\frac{4b-k-7}{5} +1) \right);$$each summand on the right-hand-side of the equality is zero as soon as $k < 4b+3$, which is consistent with our numerical assumptions.
\item[$\epsilon = 2$] As in the previous case, $4b-k-6 = 5 h + 2$, with $h = \frac{4b-k-8}{5}$: The generality assumption implies
$$\alpha_1 = \alpha_2 = \alpha_3 = \frac{4b-k-8}{5}$$ and $$\alpha_4 = \alpha_5 = \frac{4b-k-8}{5} + 1.$$This gives $h^1(\pi_*(\mathcal E)) = 0$ as soon as $k < 4b+2$, which is consistent with our numerical assumptions.
\item[$\epsilon = 3$] Same computations as above give $$\alpha_1 = \alpha_2 = \frac{4b-k-9}{5} \;\;  {\rm and} \;\; \alpha_3 = \alpha_4 = \alpha_5 = \frac{4b-k-9}{5} + 1.$$This gives $h^1(\pi_*(\mathcal E)) = 0$ as soon as $k < 4b+1$, which is consistent with our numerical assumptions.
\item[$\epsilon = 4$] Finally, in this case we have $$\alpha_1 = \frac{4b-k-10}{5} \;\;  {\rm and} \;\; \alpha_2 = \alpha_3 = \alpha_4 = \alpha_5 = \frac{4b-k-10}{5} + 1.$$This gives $h^1(\pi_*(\mathcal E)) = 0$ as soon as $k < 4b$, which is consistent with our numerical assumptions.
\end{itemize}
\end{proof}
\section{Interpretation via elementary transformations and via projective geometry}
\label{reinterpretation}
In this section, we present  interpretations of  Theorem \ref{thm:van}-(ii) in terms of both elementary transformations of vector bundles on $\FF_1$
as well as in terms of projective geometry of suitable Segre varieties.
At first sight the content of this section may look as if it is not part of the main stream of results contained in this work. Nonetheless, the exploration of consequences of Theorem \ref{thm:van} under different perspectives is useful and of general interest, as the approaches presented here shed light on intrinsic behaviors of cohomological classes of line bundles on $\FF_1.$

To discuss this, recall that from Lemma \ref{lem:h1A}, for $k \geq 2b-2$ one has $h^1(A) \neq 0$,
so \eqref{eq:h1Eh1A} does not imply the non-speciality of $\mathcal E$. Indeed, when $\mathcal E = A \oplus B$, which correspond to the trivial element of $Ext^1(B,A)$, it is clear that $\mathcal E$ has the same speciality of $A$. On the other hand, we have the following simple observation, which gives another motivation for the numerical hypotheses in Theorem \ref{thm:van} - (ii).
\begin{lemma}\label{lem:h0Bh1A} If $b \leq k \leq 4b-1$, $h^0(B) \geq h^1(A)$.
\end{lemma}

\begin{proof} From \eqref{eq:h0B}, we have that $h^0(B) = 2k-2b+5$ for any $k \geq b$. On the other hand,
$h^1(A)$ is given by \eqref{eq:h1A}.
\noindent

(i) For $b \leq k \leq 2b-3$, one trivially has $h^0(B) = 2k-2b+5 \geq 0 = h^1(A)$.

\noindent

(ii) For $k= 2b-2$, $h^0(B) = 2b+1 > 1 = h^1(A)$, since $b \geq 4$ from Assumptions \ref{ass:AB}.
\noindent

(iii) For $ k \geq 2b-1 $,  $h^0(B) = 2k-2b+5 \geq 3k-6b+6 = h^1(A)$ if and only if $k \leq 4b-1$.
\end{proof}

Because by \eqref{eq:comEAB} $B$ is always non-special, from \eqref{eq:comseq} and Lemma \ref{lem:h0Bh1A}, the non-speciality  of $\mathcal E$ is equivalent to the surjectivity  of the induced coboundary map $\partial$ as in \eqref{eq:comseq}. Therefore, from Theorem \ref{thm:van} - (ii), we have the following interesting consequence:

\begin{cor}\label{cor:van} Let $2b-2 \leq k \leq 4b-1$ and let $\mathcal E \in Ext^1(B,A)$ be a general vector bundle as in Assumptions \ref{ass:AB}. Then, the coboundary map $$\partial_e : H^0(B) \to H^1(A)$$ as in \eqref{eq:comseq}, induced by the general choice $e \in Ext^1(B,A)$ corresponding to $\mathcal E$, is surjective.
\end{cor}
\begin{proof} First of all, observe that $2b-2 > \frac{3b-3}{2}$ as $b \geq 4,$ by Assumptions \ref{ass:AB}.
Therefore, from Corollary \ref{cor:dimExt1}, the general element $e \in Ext^1(B,A) $ corresponds to an
indecomposable vector bundle $\mathcal E$ on $\FF_1$ fitting in an exact sequence as follows:
$$e:\;\;\; 0 \to A \to \mathcal E \to B \to 0.$$From Theorem \ref{thm:van}-(ii) we know that for
$2b-2 \leq k \leq 4b-1$ and $e \in Ext^1(B,A)$ general, it is $h^1(\mathcal E) = 0.$ The discussion above shows that this is
equivalent to the surjectivity of $\partial_e$.
\end{proof}
The surjectivity of $\partial_e$, for $e \in Ext^1(B,A)$ general and $2b-2 \leq k \leq 4b-1$, is strictly related
to intrinsic behaviours of some cohomological classes of line bundles on $\FF_1$. Indeed, let $A$ and $B$ be as in \eqref{eq:al-be3}. One has a natural cup-product between cohomological classes:
\begin{equation}\label{eq:cup}
H^0(B) \otimes H^1(A-B) \stackrel{\cup}{\longrightarrow} H^1(A), \;\; \sigma \otimes e \to \sigma \cup e,
\end{equation}for any $\sigma \in H^0(B)$ and any $e \in H^1(A-B)$. The cup product induces natural linear maps; precisely, for any fixed $e \in H^1(A-B)$, one has
\begin{equation}\label{eq:cupe}
H^0(B) \stackrel{ - \cup e}{\longrightarrow} H^1(A), \;\; \sigma  \to \sigma \cup e,
\end{equation}whereas, for any fixed $\sigma \in H^0(B)$, one has
\begin{equation}\label{eq:cups}
H^1(A-B) \stackrel{ \sigma \cup -}{\longrightarrow} H^1(A), \;\; e  \to \sigma \cup e.
\end{equation}

The canonical isomorphism $H^1(A-B) \cong Ext^1(B,A)$ implies that, for any $e \in Ext^1(B,A)$, $$\partial_e = - \cup e.$$ Similarly, let us set $$\Phi_{\sigma} = \sigma \cup -.$$ Following \cite[p. 31]{Frie}, $Ext^1(B,A)$ parametrizes {\em strong isomorphism classes} of extensions of line bundles. Therefore, given $ e \in Ext^1(B,A)$, it corresponds to an extension class as in \eqref{eq:al-be}. It is clear that $ e= 0$ corresponds to $A\oplus B$. Since  $H^1(A) \cong Ext^1(\Oc, A)$, then $\partial_e(\sigma) = \sigma\cup e$ corresponds to an extension class of $\Oc$ with $A$.

By \cite[pp. 41-42]{Frie} (cf. also \cite{Bo} and \cite{H-L}),  the maps
$\partial_e$ and $\Phi_{\sigma}$ can be read in terms of {\em elementary transformations} on $\FF_1$.
Precisely, fix $e \in Ext^1(B,A)$; given $0 \neq \sigma \in H^0(B)$ and using the exact sequence \eqref{eq:al-be}, we have the following diagram:
\begin{displaymath}
\begin{array}{rcccccc}
      &   &     &            &     & 0 & \\
              &   &     &            &     & \downarrow & \\
           &   &     &            &     & \Oc & \\
      &   &     &            &     & \downarrow^{\sigma} & \\
(e): \;\;\;0 \to & A & \to & \mathcal E & \to & B & \to 0 \\
     &   &     &            &     & \downarrow & \\

     &   &     &            &     & N_{C_{\sigma}/\FF_1} & \\
          &   &     &            &     & \downarrow & \\
         &   &     &            &     & 0 &
\end{array}
\end{displaymath}where $C_{\sigma} \in |B|$ is the vanishing locus of
$\sigma \in H^0(B)$, and $N_{C_{\sigma}/\FF_1} \cong \Oc_B(B) $. By composition, we have a surjective morphism  $\mathcal E \to\!\!\!\! \to \Oc_B(B)$ of sheaves on $\FF_1$; thus the previous diagram can be completed as follows:
 \begin{displaymath}
\begin{array}{rcccccc}

      &   &     &            &     & 0 & \\

              &   &     &            &     & \downarrow & \\

           &   &     &            &     & \Oc & \\

      &   &     &            &     & \downarrow^{\sigma} & \\

(e): \;\;\;0 \to & A & \to & \mathcal E & \to & B & \to 0 \\

     &   &     &       \downarrow^{\psi}    &     & \downarrow & \\

     &   &     0 \to &       \Oc_B(B)    &  \stackrel{id}{\longrightarrow}   & \Oc_B(B) & \to 0 \\

          &   &     &     \downarrow       &     & \downarrow & \\

          &   &     &           0 &     & 0 &

\end{array}
\end{displaymath}
 Let $\mathcal W_{\sigma} := ker(\psi)$. Since $C_{\sigma}$ is an effective divisor, $\Oc_{C_{\sigma}}(B) \cong \Oc_B(B)$ is a line-bundle on $C_{\sigma}$. If $\mathcal E$ is a rank-two vector bundle, then from \cite[Lemma 16, p. 41]{Frie}
$\mathcal W_{\sigma}$ is also locally free, of rank-two on $\FF_1$; its Chern classes are
$$c_1(\mathcal W_{\sigma}) = c_1(\mathcal E) - [C_{\sigma}], \;\; c_2(\mathcal W_{\sigma}) = c_2(\mathcal E) - c_1(\mathcal E) \cdot [C_{\sigma}] + j_*(\Oc_{C_{\sigma}}(B)),$$
where $[C_{\sigma}]\in H^2(\FF_1, \mathbb{Z})$ denotes the cycle defined by $C_{\sigma}$ and $j : C_{\sigma} \hookrightarrow \FF_1$ the natural inclusion. In particular, from Assumptions \ref{ass:AB}, $$ c_1(\mathcal W_{\sigma}) \equiv A \;\; {\rm and} \;\; c_2(\mathcal W_{\sigma})= 0.$$ This is in accordance with the fact that the previous
diagram can clearly be further completed as follows:
\begin{equation}\label{eq:diag}
\begin{array}{rcccccc}
      &  0 &     &    0        &     & 0 & \\

      & \downarrow  &     &     \downarrow       &     & \downarrow & \\

 0 \to   &  A &    \to  &         \mathcal W_{\sigma}   &   \to  & \Oc & \to 0 \\

      &  \downarrow &     &       \downarrow     &     & \downarrow^{\sigma} & \\

(e): \;\;\;0 \to & A & \to & \mathcal E & \to & B & \to 0 \\

     &   \downarrow &     &       \downarrow     &     & \downarrow & \\

     &   0 &      \to &       \Oc_B(B)    &  \stackrel{id}{\longrightarrow}   & \Oc_B(B) & \to 0 \\

          &   &     &     \downarrow       &     & \downarrow & \\

          &   &     &           0 &     & 0 & .

\end{array}
\end{equation} The above diagram  describes $\partial_e(\sigma) \in H^1(A)$ as
$$\partial_e(\sigma) : \;\;\; 0 \to A \to  \mathcal W_{\sigma} \to \Oc \to 0 .$$Note that, if $\partial_e(\sigma) = 0 \in H^1(A)$ - i.e. $\sigma \in ker(\partial_e)$ -
then $\mathcal W_{\sigma} = A \oplus \Oc$. On the other hand, if $\sigma \in \!\!\!\!\!| \;\; ker(\partial_e)$,
then $\mathcal W_\sigma$ is a non-trivial extension class in $Ext^1(\Oc,A)$.

Similarly, to describe $\Phi_{\sigma}$, for a fixed $0 \neq \sigma \in H^0(B)$, one has an exact sequence as the right-hand-side column of diagram \eqref{eq:diag}. Therefore for any $e \in Ext^1(B,A)$, which gives rise to an exact sequence $(e)$ as the middle row of diagram \eqref{eq:diag}, $\Phi_{\sigma}(e)$ is defined as the first row of diagram \eqref{eq:diag}.

The discussion conducted above implies that $\Phi{\sigma}$ is surjective, for any $\sigma \in H^0(B)$. Indeed, from \eqref{eq:h0B} $B$ is effective, thus the exact sequence defining $B$ in $\FF_1$ tensored by $A$ gives
$$0 \to A-B \to A \to \Oc_B(A) \to 0.$$Now $\Oc_B(A) \cong \Oc_{\Pp^1}(k)$, since $B$ is rational being a unisecant of $\FF_1$. Moreover,
$AB = c_2(\mathcal E) = k \geq 2b-2 \geq 6$ by assumption, thus $h^1(\Oc_B(A)) = 0$, which implies the surjectivity of $\Phi_{\sigma}$.
Since $H^1(A-B) \cong Ext^1(B,A)$ and $H^1(A) \cong Ext^1(\Oc, A)$ canonically, the associated cohomology sequence gives
$$Ext^1(B,A) \stackrel{\Phi_{\sigma}}{\to} Ext^1(\Oc, A) \to 0,$$for any $\sigma \in H^0(B)$; in other words, for any $\sigma \in H^0(B)$, any extension class in $Ext^1(\Oc, A)$ is obtained as an elementary transformation along the curve $C := C_{\sigma}$ of some extension in  $Ext^1(B,A)$. In particular  when $k \geq 2b-1$, from Lemma \ref{lem:h1A}-(iii), $A$ is not effective, so $ker(\Phi_{\sigma}) \cong H^0(\Oc_B(A)) \cong \mathbb{C}^{k+1}$, for any $\sigma \in H^0(B)$. In other words,
for any $v \in Ext^1(\Oc, A)$, there is a $(k+1)$-dimensional family of extensions in $Ext^1(B,A)$ inducing $v$ via $\Phi_{\sigma}$.

Similarly, Corollary \ref{cor:van} implies that, when $k \geq b$,  as soon as $h^0(B) \geq h^1(A)$ and  $e \in Ext^1(B,A)$ is general, any  extension class in $Ext^1(\Oc, A)$ is obtained as an elementary transformation of $\mathcal E$ along some divisor $C_{\sigma} \subset \FF_1$, for some $\sigma \in H^0(B)$. In particular  when $k \geq 2b-1$, from Lemma \ref{lem:h1A}-(iii), $A$ is not effective, so $ker(\partial_e) \cong H^0(\mathcal E) \subset H^0(B)$. In other words, for a general $e,$ for any $v \in Ext^1(\Oc, A)$ there is a sub-linear system $\Lambda \subset |B|$ of curves
on $\FF_1$ of dimension $h^0(\mathcal E) -1$, inducing the same extension class $v$ via $\partial_e$.

Other independent, interesting consequences of Corollary \ref{cor:van} can be highlighted.
As already observed, $\partial_e$ is induced by the natural cup-product \eqref{eq:cup}. This
can be interpreted via linear projections of suitable projective varieties. Indeed, let
$$\Sigma = \Pp^{2k-2b+4} \times \Pp^{4k-6b+6} = \Pp(H^0(B)) \times \Pp(H^1(A-B))$$be the corresponding {\em Segre variety} and let $M := h^0(B) h^1(A-B) - 1$. Denote by $p_{\partial}$ the projectivization of the cup-product \eqref{eq:cup}; thus  $p_{\partial}$ is the composition
\begin{displaymath}
\begin{array}{ccc}
\Sigma & \hookrightarrow & \Pp^M \\

    & \searrow^{p_{\partial}} & \downarrow^{\Pi} \\

    &    & \Pp^{3k-6b+5}
\end{array}
\end{displaymath}where $\Pp^{3k-6b+5} = \Pp(H^1(A))$ and $\Pi$ is a linear projection. Let $\Xi \subset \Pp^M$ be the center of the projection
$\Pi$, thus $\Xi \cong \Pp^{M - 3k + 6b - 6}$. What Theorem \ref{thm:van} and Corollary \ref{cor:van} establish is that $\Xi$ intersects
the general ruling $\Pp^{2k-2b+4}$ of the Segre variety $\Sigma$ in the expected dimension $4b-2-k$; in other words, the restriction  $\Pi|_{\Pp^{2k-2b+4}}$ to the general ruling $\Pp^{2k-2b+4}$ of $\Sigma$ is surjective onto $\Pp^{3k-6b+5}$.

\section{$3$-dimensional scrolls over $\FF_1$ and their Hilbert schemes} \label{S:scrollsF1}
In this section, results from \S\,\ref{S:vbF1} are applied to the study of suitable $3$-dimensional
scrolls over $\FF_1$ in projective spaces and some components of their Hilbert schemes.

Assume from now on that $\mathcal E$ is a very-ample rank-two vector bundle on $\FF_1$. Notice that
the choice of the numerical class of $c_1(\mathcal E)$ together with the very-ampleness hypothesis, naturally lead to
Assumptions \ref{ass:AB}. Indeed, recall the following necessary condition for very-ampleness:
\begin{prop}\label{prop:AB} (see \cite[Prop. 7.2]{al-be}) Let $\mathcal E$ be a very-ample, rank-two vector bundle over $\FF_1$ such that
$$c_1(\mathcal E) \num 3 C_0 + b f \; \; {\rm and} \;\; c_2(\mathcal E) = k.$$Then $\mathcal E$
satisfies all the hypotheses in Assumptions \ref{ass:AB} and moreover
\begin{equation}\label{eq:bounds}
h^0(\FF_1 , \mathcal E) := n+1 \geq 7.
\end{equation}
\end{prop}
 A few remarks are in order. First of all, from Corollary \ref{cor:dimExt1}-(i), when $b \leq k < \frac{3b-3}{2}$, $\mathcal E$ splits as $\mathcal E = A \oplus B$; therefore $\mathcal E$ is very ample if and only if both  line bundles $A$, $B$ as in Assumptions \ref{ass:AB} are very-ample.  From \cite[Thm. 2.17, p. 379, and Cor. 2.18, p. 380]{H} a sufficient condition for the very-ampleness of both $A$ and $B$ is $b \geq 5$.

Secondly, as \eqref{eq:bounds} requires $h^0(\mathcal E) \geq 7$, from \eqref{eq:h0E} and Theorem \ref{thm:van}, we will also assume $b \leq k \leq 4b-8$.

In other words, from now on Assumptions \ref{ass:AB} are replaced by:
\begin{ass}\label{ass:AB2} Let $\mathcal E$ be a very-ample, rank-two vector bundle over
$\FF_1$ such that $$c_1(\mathcal E) \num 3 C_0 + b f, \;\; c_2(\mathcal E) = k, \;\; {\rm with} \;\;\; 5 \leq b \leq k \leq 4b-8.$$
\end{ass}
Under these assumptions it follows that $ \mathcal E$
fits in an exact sequence as
\begin{equation}\label{eq:al-beb}
0 \to A \to \mathcal E \to B \to 0,
\end{equation}where $A$ and $B$ are such that
\begin{equation}\label{eq:al-be3b}
A \num 2 C_0 + (2b-k-2) f \;\; {\rm and} \;\; B \num C_0 + (k - b + 2) f.
\end{equation}
see for instance \cite[Proof of Proposition 7.2]{al-be}.

With this set up, let $\scrollcal{E}$ be the associated 3-dimensional scroll over $\FF_1,$
and let $\pi: \FF_1 \to \Pin{1}$ and $\varphi: \Pp(\mathcal E) \to  \FF_1$ be the usual projections.
Denote by $L := \Oc_{\Pp(\mathcal E)}(1)$ its tautological line-bundle.

\begin{prop}\label{prop:X} Let $\mathcal E$ be as in Assumptions \ref{ass:AB2}. Moreover,
when $\frac{3b-3}{2} \leq k \leq 4b-8$ we further assume that  $\mathcal E \in Ext^1(B,A)$ is general.
Then $L$ defines an embedding
\begin{equation}\label{eq:X}
\Pp(\mathcal E) \stackrel{|L|}{\longrightarrow} X \subset \Pin{n},
\end{equation} where $X$ is a smooth, projective  $3$-fold scroll over $\FF_1$, non-degenerate in
$\Pp^n$, of degree $d$ with
\begin{equation}\label{eq:nd}
n = 4b-k-2 \geq 6 \;\;\; {\rm and} \;\;\; d = 6b-9-k
\end{equation}and such that
\begin{equation}\label{eq:van}
h^i(X, L) = 0, \;\; i \geq 1.
\end{equation}
\end{prop}
\begin{proof} By Assumptions \ref{ass:AB2}, the very-ampleness of $\mathcal E$ is equivalent to that of $L$.

The formula on the degree of $X$ follows from \eqref{eq:d} and \eqref{eq:al-be3b}.

Condition \eqref{eq:van} follows from Leray's isomorphisms, Proposition \ref{prop:comEAB} and Theorem
\ref{thm:van}.

Finally, since $n+1 = h^0(X,L) = h^0(\FF_1, \mathcal E)$, then $n+1 \geq 7$ (as in \eqref{eq:bounds}) follows from
\eqref{eq:h0E}, Theorem \ref{thm:van} and the fact that $k \leq 4b-8$.
\end{proof}

In what follows, we will be interested in studying the Hilbert scheme parametrizing subvarieties of $\Pp^n$ having the same Hilbert polynomial of $X$.

\subsection{Basics on Hilbert Scheme}\label{S:hilb}
The existence of the Hilbert scheme parametrizing closed subschemes of $\Pin{n}$
with given Hilbert polynomial was established by Grothendieck, \cite{groth}. The following formulation of his basic result is due to
Sommese, \cite{SO1}.
\begin{prop}[\cite{groth}, \cite{SO1}]
\label{basic fact} Let $Z$ be a smooth connected projective variety.
Let $X$ be a connected subvariety which is a local complete intersection in $Z$ and with $H^1(X,N)=0$ where
$N := N_{X/Z}$ is the normal bundle of $X$ in $Z$.
Then there exist irreducible projective varieties ${\mathcal Y}$ and
${\mathcal H}$ with the following properties:
\begin{itemize}
\label{properties}
  \item[(i)]
${\mathcal Y} \subset {\mathcal H}\times Z$ and the map
$p:{\mathcal Y}\lra {\mathcal H}$ induced
by the product projection is a flat surjection,
  \item[(ii)]  there is a smooth point $x\in {\mathcal H}$
with $p$ of maximal rank in a neighborhood of $p^{-1}(x)$,
  \item[(iii)]  $q$ identifies  $p^{-1}(x)$ with $X$ where
$q:{\mathcal Y}\lra Z$ is the map induced
by the product projection, and
\item[(iv)] $H^0(N)$ is naturally identified with $T_{{\mathcal H},x}$
where
$T_{{\mathcal H},x}$ is the Zariski tangent space of ${\mathcal H}$ at
$x$.
\end{itemize}
\end{prop}
Recall that, when $x \in {\mathcal H}$ is a smooth point, the corresponding subvariety $X_x \subset Z$ parametrized by $x$ is said to be {\em unobstructed} in $Z$.

\subsection{The irreducible component $\mathcal X$ of the Hilbert scheme containing [X]}\label{S:hilbX}
The scroll $X \subset \Pp^n$ as in Proposition \ref{prop:X} corresponds to a point $[X] \in \mathcal H$, where $\mathcal H$ denotes the Hilbert scheme parametrizing $3$-dimensional subvarieties of $\Pp^n$, of degree $d = 6b - k - 9$.
The next result shows that $X$ is unobstructed in $\Pp^n$.

\begin{prop}\label{Hilbertscheme of scrollover F1} There exists an irreducible component $\mathcal X \subseteq \mathcal H$, which is generically smooth and of (the expected) dimension
\begin{equation}\label{eq:expdim}
\dim(\mathcal X) =  n(n+1) + 3k - 2b - 2
\end{equation}such that $[X]$ belongs to the smooth locus of $\mathcal X$.
\end{prop}
\begin{proof} Let $N : = N_{X/\Pp^n}$ denote the normal bundle of $X$ in ${\mathbb P}^n$.
The statement will follow from Proposition \ref{basic fact} by showing that $H^i(X,N)=0$, $i \geq 1$, and conducting an
explicit computation of $h^0(X,N) = \chi(X, N).$

To do this, let
\begin{eqnarray}\label{eulersequscrollsuF1}
0\lra {\mathcal O}_{X} \lra {\mathcal O}_{X}(1)^{\oplus (n+1)} \lra T_{{{\mathbb P}^n}|{X}} \lra  0
\end{eqnarray}be the Euler sequence on ${\mathbb P}^n$ restricted to $X$. Since $(X,L)$ is a scroll over $\FF_1$,
\begin{eqnarray}
\label{quadratino}
H^{i}(X,{\mathcal O}_{X})= H^{i}(\FF_1,{\mathcal O}_{\FF_1})= 0,\quad \text{ for}\quad i\ge 1.
\end{eqnarray}

Moreover, by \eqref{eq:van}, we have
\begin{eqnarray}
\label{doppioquadratino}
H^i(X,L)= 0, \;\; i \geq 1.
\end{eqnarray}
Thus, from \brref{quadratino},
\brref{doppioquadratino} and the cohomology sequence associated to
\brref{eulersequscrollsuF1} it follows that  $H^i(X,T_{{ {\mathbb
P}^n}|{X}})=0$
for $i\ge 1$.

Therefore the exact sequence
\begin{eqnarray}
\label{tangentsequ} 0\lra T_{X} \lra T_{{ {\mathbb P}^n}|{X}}
\lra N \lra 0
\end{eqnarray}
gives
\begin{eqnarray}
\label{cohomnormal}
H^{i}(X,N) \cong H^{i+1}(X,T_{X}) \qquad {\text {for} \quad  i\ge 1.}
\end{eqnarray}
\begin{claim}\label{cl:clar} $H^i(X, N) = 0$, for $i \geq 1$.
\end{claim}

{\it Proof of Claim \ref{cl:clar}} \, It is obvious that $H^3(X,N) = 0$, for dimension reasons. For the other cohomology spaces, we can use
\eqref{cohomnormal}.

In order to compute $H^{j}(X,T_{X}),\; j = 2,3,$ we use the scroll map
$\varphi:{\mathbb P}({\mathcal E})\lra \FF_1$ and we consider the relative cotangent bundle sequence:
\begin{eqnarray}\label{relativctgbdl}
0\to \varphi^{*}({\Omega}^1_{\FF_1})\to {\Omega}^1_{X}
\to {\Omega}^1_{X|{\FF_1}} \lra 0.
\end{eqnarray}
From \brref{relativctgbdl} and the Whitney sum, one obtains
$$ c_1({\Omega}^1_{X})= c_1(\varphi^{*}({\Omega}^1_{\FF_1}))+c_1({\Omega}^1_{X|{\FF_1}})$$ and thus
$${\Omega}^1_{X|{\FF_1}}=K_X+\varphi^{*}(-c_1({\Omega}^1_{\FF_1}))=K_X+\varphi^{*}(-K_{\FF_1}).$$
The adjunction theoretic characterization of the scroll gives
$$K_X= -2L+\varphi^{*}(K_{\FF_1}+c_1({\mathcal E}))=-2L+\varphi^{*}(K_{\FF_1}+3C_{0}+bf)$$
thus$${\Omega}^1_{X|{\FF_1}}=K_X+\varphi^{*}(-K_{\FF_1})=
-2L+\varphi^{*}(3C_{0}+bf)$$which, combined with the dual of \brref{relativctgbdl}, gives
\begin{eqnarray}
\label{relative tgbdl}
0 \to 2L-\varphi^{*}(3C_{0}+bf) \to T_X
\to \varphi^{*}(T_{\FF_1}) \to 0.
\end{eqnarray}

\noindent

(i) First of all, we compute the cohomology of $\varphi^{*}(T_{\FF_1})$. Consider the relative cotangent bundle sequence of the map
$\pi: \FF_1 \to \Pin{1}$
\begin{eqnarray}
\label{relativctgbdlF1}
0\to \pi^{*}{\Omega}^1_{\Pin{1}} \to {\Omega}^1_{\FF_1}
\to {\Omega}^1_{\FF_1|\Pin{1}} \to 0.
\end{eqnarray}Since $ {\Omega}^1_{\FF_1|\Pin{1}} =K_{\FF_1}+ \pi^{*} \oofp{1}{2}=-2C_{0}-f$, dualizing \brref{relativctgbdlF1} we get
\begin{eqnarray}
\label{1relativctgbdlF1}
0\to 2C_{0}+f \to T_{\FF_1}
\to\pi^{*}T_{\Pin{1}} \to 0.
\end{eqnarray}
From the cohomology sequence associated to \brref{1relativctgbdlF1} we get that $H^{i}(\FF_1,  T_{\FF_1})=0$, for any
$i \ge 1$. By Leray's exact sequence, the same holds for $H^{i}(X,  \varphi^*(T_{\FF_1}))=0$, $i \geq 1$.

\noindent

(ii) We now devote our attention to the cohomology of
$2L-\varphi^{*}(3C_{0}+bf).$ Noticing that  $R^{i}\varphi_{*}(2L)=0$, for $i \ge 1$
(see \cite{H}, p. 253), projection formula and Leray's spectral
sequence give $$H^i(X, 2L-\varphi^{*}(3C_{0}+bf)) \cong
H^i({\FF_1}, Sym^2{\mathcal E}\otimes (-3C_{0}-bf)).$$
Therefore
\begin{equation}
\label{h3sym2E}
H^3(X, 2L-\varphi^{*}(3C_{0}+bf)) = 0
\end{equation} for dimension reasons.

We need to show that $H^2({\FF_1}, Sym^2{\mathcal E}\otimes (-3C_{0}-bf))=0.$ We write $-3C_{0}-bf$ as $-2(C_{0}+2f)-(C_0+(b-4)f),$
 hence $$Sym^2{\mathcal E}\otimes (-3C_{0}-bf)=Sym^2{\mathcal E}\otimes (-C_{0}-(b-4)f)\otimes (-2H),$$where $H=C_{0}+2f$ is
a very-ample line bundle on $\FF_1$ (cf. \cite[Thm. 2.17-(c), p.379]{H}). Let $$\mathcal {S}:=Sym^2{\mathcal E}\otimes (-C_{0}-(b-4)f);$$then we need to show that  $H^2({\FF_1}, {\mathcal S}\otimes (-2H))=0.$

We let  $\mathcal {S}(t)$ denote $Sym^2{\mathcal E}\otimes (-C_{0}-(b-4)f)\otimes (tH).$
We write
\begin{eqnarray*}
\mathcal {S}(t)&=&Sym^2{\mathcal E}\otimes (-C_{0}-(b-4)f)\otimes (tH)\\
&=&Sym^2{\mathcal E}\otimes (-C_{0}-(b-4)f)\otimes (tC_0+2tf)={\mathcal S}(t, 2t).\nonumber
\end{eqnarray*}
Let us consider the structure sequence of $f$ and $C_{0}$ on $\FF_1$, respectively
\begin{equation}
\label{structureseqf}
0 \to -f \lra {\mathcal O}_{\FF_1} \to
\restrict{{\mathcal O} }{f}  \to 0
\end{equation}
\begin{equation}
\label{structureseqC0}
0 \to -C_{0} \lra {\mathcal O}_{\FF_1} \to
\restrict{{\mathcal O} }{C_{0}}  \to 0
\end{equation}
Tensoring \brref{structureseqf} with ${\mathcal S}(t, 2t)$ gives
\begin{equation}
\label{mathcalS}
0 \to {\mathcal S}(t, 2t-1) \lra {\mathcal S}(t,2t) \to
\restrict{{\mathcal S}(t,2t)}{f} \to 0
\end{equation}
As $c_{1}({\mathcal E})=3C_{0}+bf,$  and  ${\mathcal E}$ is very-ample, the splitting type of
$\mathcal{E}$ on any fibre $f$  is, $\restrict{{\mathcal E}}{f}=\oofp{1}{2}\oplus \oofp{1}{1}.$  Hence we get that
\begin{eqnarray*}
\restrict{{\mathcal S}}{f}&=&\restrict{(Sym^2{\mathcal E}\otimes (-C_{0}-(b-4)f))}{f}=\restrict{Sym^2{\mathcal E}}{f}\otimes \oofp{1}{-1}\\
&=&\oofp{1}{1}\oplus \oofp{1}{2} \oplus \oofp{1}{3}.
\end{eqnarray*}
Let us determine $\restrict{{\mathcal S}}{C_{0}}.$
Recall that  $(3C_{0}+bf)\cdot C_{0}=-3+b\ge 2$. Suppose that $\restrict{{\mathcal E}}{C_{0}}=\oofp{1}{\alpha}\oplus \oofp{1}{\beta},$ where $\alpha+\beta=b-3$ , $\alpha\ge 1$ and $\beta\ge 1$, by ampleness. We may assume that
$\alpha\ge \beta$ from which it follows that $2\alpha\ge b-3$ and $2\beta\le b-3$. Hence  $\restrict{(Sym^2{\mathcal E})}{C_{0}}=\oofp{1}{2\alpha}\oplus \oofp{1}{2\beta} \oplus \oofp{1}{\alpha+\beta}$ and thus
\begin{eqnarray*}
\restrict{{\mathcal S}}{C_{0}}&=&\restrict{(Sym^2{\mathcal E}\otimes (-C_{0}-(b-4)f))}{C_{0}}=\restrict{Sym^2{\mathcal E}}{C_{0}}\otimes \oofp{1}{5-b}\\
&=&\oofp{1}{2\alpha+(5-b)}\oplus \oofp{1}{2\beta+(5-b)} \oplus \oofp{1}{2}.
\end{eqnarray*}

From \brref{mathcalS} in order to have $H^2({\FF_1}, \mathcal {S}(t, 2t))=H^2({\FF_1}, \mathcal {S}(t, 2t-1))$ one needs  $H^1(\restrict{{\mathcal S}(t,2t)}{f})=0.$

Note that $H^1(\restrict{{\mathcal S}(t,2t)}{f} )=H^1( (\oofp{1}{1}\oplus \oofp{1}{2} \oplus \oofp{1}{3})\otimes \oofp{1}{t})=H^1( \oofp{1}{t+1}\oplus \oofp{1}{t+2} \oplus \oofp{1}{t+3})=0$ if $t+1 >-2$ that is if $t >-3.$ Hence for $t >-3$

\begin{equation}
\label{1mathcalStwist1}
 H^2({\FF_1}, \mathcal {S}(t, 2t))=H^2({\FF_1}, \mathcal {S}(t, 2t-1))
\end{equation}

Tensoring \brref{structureseqf} with ${\mathcal S}(t, 2t-1)$ gives
\begin{equation}
\label{mathcalStwist1}
0 \to {\mathcal S}(t, 2t-2) \lra {\mathcal S}(t,2t-1) \to
\restrict{{\mathcal S}(t,2t-1)}{f} \to 0
\end{equation}
Since  $H^1(\restrict{{\mathcal S}(t,2t-1)}{f} )=H^1( (\oofp{1}{1}\oplus \oofp{1}{2} \oplus \oofp{1}{3})\otimes \oofp{1}{t})=H^1({\FF_1}, \oofp{1}{t+1}\oplus \oofp{1}{2t+} \oplus \oofp{1}{t+3})=0$ if  $t >-3,$ hence for $t >-3$
\begin{equation}
\label{2mathcalStwist1}
 H^2({\FF_1}, \mathcal {S}(t, 2t-1))=H^2({\FF_1}, \mathcal {S}(t, 2t-2))
\end{equation}
Tensoring \brref{structureseqC0} with ${\mathcal S}(t, 2t-2)$ gives
\begin{equation}
\label{mathcalStwist2}
0 \to {\mathcal S}(t-1, 2t-2) \lra {\mathcal S}(t,2t-2) \to
\restrict{{\mathcal S}(t,2t-2)}{C_{0}} \to 0
\end{equation}
We would like to have $H^2({\FF_1}, \mathcal {S}(t, 2t-2))=H^2({\FF_1}, \mathcal {S}(t-1, 2t-2)).$ In order for this to happen we need  $H^1(\restrict{{\mathcal S}(t,2t-2)}{C_{0}})=0.$

But{\small
$$H^1({\mathcal S}(t,2t-2)_{|{C_{0}}})=
H^1((\oofp{1}{2\alpha+5-b}\oplus \oofp{1}{2\beta+5-b} \oplus \oofp{1}{2})\otimes \oofp{1}{t-2})
=0,$$
}
if  $2\beta+3-b+t > -2$, that is $t > -2$. Thus for $t > -2$
\begin{equation}
\label{1mathcalStwist2}
H^2({\FF_1}, \mathcal {S}(t, 2t-2))=H^2({\FF_1}, \mathcal {S}(t-1, 2t-2)
\end{equation}
Combining \brref{1mathcalStwist1}, \brref{2mathcalStwist1}, \brref{1mathcalStwist2} we have
$H^2({\FF_1}, \mathcal {S}(t, 2t))=H^2({\FF_1}, \mathcal {S}(t-1, 2t-2)$
for all $t \ge -1.$ Therefore Serre's vanishing theorem gives
$$H^2({\FF_1}, \mathcal {S}(t, 2t))=0 , \quad \text{for all } t \ge -1$$
and thus  also
$$H^2({\FF_1}, \mathcal {S}(t-1, 2t-2))=0 , \quad \text{for all } t \ge -1$$
In particular, if $t=-1$, $H^2({\FF_1},\mathcal {S}(-2, -4))=0$.  Note that
$H^2({\FF_1},\mathcal {S}(-2, -4))=H^2({\FF_1},Sym^2{\mathcal E}\otimes (-3C_{0}-bf))$ and hence $H^2({\FF_1},Sym^2{\mathcal E}\otimes (-3C_{0}-bf))=0.$

Thus, from  the cohomology
sequence associated to \brref{relative tgbdl},\brref{h3sym2E}, and
dimension reasons, it follows that
$H^2(X, T_X)= H^2(X,  \varphi^*T_{\FF_1})$ and $H^3(X, T_X) = 0.$
On the other hand, by Leray spectral sequence,
$$H^2(X, \varphi^*T_{\FF_1}) = H^2(\FF_1, T_{\FF_1})=0.$$
Hence $H^i(X,T_{X})=0, i= 2,3$ and thus, by \brref{cohomnormal},
$H^i(X,N)=0 \quad i = 1,2,3$, which concludes the proof of Claim \ref{cl:clar}.

To conclude the proof of Proposition \ref{Hilbertscheme of scrollover F1}, according to Proposition \ref{basic fact},
we deduce there exists an irreducible component $\mathcal{X}$ of the Hilbert scheme $\mathcal H$ containing $[X]$ as a smooth point.
Since smoothness is an open condition, we deduce that $\mathcal X$ is generically smooth.
The dimension of ${\mathcal X}$, by Proposition \ref{basic fact} - (iv) and Claim \ref{cl:clar}, will be given by $h^0(X,N)= \chi(N)$, where $\chi(N)$ is the {\em expected dimension} of $\mathcal H$.

The Hirzebruch-Riemann-Roch theorem gives

\begin{eqnarray}
\label{chiN}
\chi(N) &=& \frac{1}{6}(n_1^3-3n_1n_2+3n_3)+
\frac{1}{4}c_1(n_1^2-2n_2)+\frac{1}{12}(c_1^2+c_2)n_1 \\
& & +(n-3)\chi({\mathcal O}_{X})\nonumber
\end{eqnarray}

where $n_i=c_i(N),$ and $ c_i=c_i(X).$

Chern classes of $N$  can be obtained from
\brref{tangentsequ}:
\begin{xalignat}{1}
\label{valueniscrollP}
n_1=&K+(n+1)L;  \\
n_2=&\frac{1}{2}n(n+1)L^2+(n+1)LK+K^2-c_2; \notag\\
n_3=&\frac{1}{6}(n-1)n(n+1)L^3+\frac{1}{2}n(n+1)KL^2+
(n+1)K^2L-(n+1)c_2L \notag \\
&-2c_2K+K^3-c_3. \notag
\end{xalignat}

The numerical invariants of $X$ can be easily computed:

\begin{xalignat}{2}
  KL^2 &= -2d+4b-12;&  K^2L &= 4d-14b+41;\notag \\
  c_2L &= 2b+7;& K^3 &= -8d+36b-102;\notag \\
  -Kc_2 &= 24;&  c_3 &= 8. \nonumber
\end{xalignat}

Plugging these in \brref{valueniscrollP} and the results in
\brref{chiN} one gets
$$\chi(N)=(-2b+8+d)n-29+16b-3d.$$By \eqref{eq:d} and Proposition \ref{prop:AB}, we can express $d$
just in terms of $b$ and $k$. Moreover, from \eqref{eq:h0E} and from Theorem \ref{thm:van}, one has $n+1 = 4b-k-1$ which gives exactly \eqref{eq:expdim}.
\end{proof}

\subsection{The general point of $\mathcal X$}\label{S:genpoint} In this subsection a description of  the general point of the component $\mathcal X$ determined in Proposition \ref{Hilbertscheme of scrollover F1} is presented.

\begin{theo}\label{thm:genpo} The general point of $\mathcal X$ parametrizes a scroll $X$ as in Proposition \ref{prop:X}.
\end{theo}
\begin{proof} The proof reduces to a parameter computation to obtain a lower bound for $\dim(\mathcal X)$. Precisely, we shall first compute a lower bound for the dimension of the locus in $\mathcal X$ filled-up by scrolls of type $(X,L)$ as in Proposition \ref{prop:X}. Comparing  this lower bound with \eqref{eq:expdim} will conclude the proof.

From the exact sequence \eqref{eq:al-beb}, we observe that
\begin{itemize}
\item[a)]the line bundle $A$ is uniquely determined on $\FF_1$, since $A \cong \Oc_{\FF_1}(C_0) \otimes \pi^*(\Oc_{\Pp^1}(2b-k-2))$;
\item[b)] the line bundle $B$ is uniquely determined on $\FF_1$, similarly.
\end{itemize}
Therefore, let $\mathcal Y \subseteq \mathcal X$ be the locus filled-up by scrolls $X$ as in Proposition \ref{prop:X}. Let us compute
how many parameters are needed to describe $\mathcal Y$. To do this, we have to add up the following quantities:
\begin{itemize}
\item[1)] $0$ parameters for the the line bundle $A$ on $\FF_1$, since it is uniquely determined;
\item[2)] $0$ parameters for the the line bundle $B$, similarly;
\item[3)] the number of parameters counting the isomorphism classes of projective bundles $\Pp(\mathcal E).$
According to Lemma \ref{lem:ext1} this number is
\begin{equation}\label{eq:dimExt1b}
\tau: = \left\{\begin{array}{ccc}
0 & & b \leq k <  \frac{3b-3}{2}\\
4k-6b+6 & & \frac{3b-3}{2} \leq k \leq 4b-8.
\end{array}
\right.
\end{equation}Indeed, for $k<\frac{3b-3}{2}$ the only vector bundle is $\mathcal E = A \oplus B$. For $k \geq \frac{3b-3}{2}$, the general vector
bundle $\mathcal E$ is indecomposable and non-special (cf. Theorem \ref{thm:van}); in this latter case we have to take into account
{\em weak isomorphism classes} of extensions, which are parametrized by $\Pp(Ext^1(B,A))$ (cf. \cite[p. 31]{Frie}).
\item[4)]  $$\dim(PGL(n+1, \mathbb C)) - \dim (G_X),$$ where $G_X \subset PGL(n+1, \mathbb C)$ denotes the {\em stabilizer} of $X \subset \Pp^n$, i.e. the subgroup of projectivities of $\Pp^n$ fixing $X$. In other words, $\dim(PGL(n+1, \mathbb C)) - \dim (G_X)$ is the dimension of the full orbit of $X \subset \Pp^n$ by the action of all the projective transformations of $\Pp^n$.
\end{itemize}
Since $\dim(PGL(n+1, \mathbb C)) = (n+1)^2-1 = n(n+2)$, the previous computation shows that
\begin{equation}\label{eq:dimY}
\dim(\mathcal Y) = \tau + n(n+2) - \dim(G_X)
\end{equation}The next step is to find an upper bound for $\dim(G_X)$.

It is clear that there is an obvious inclusion
\begin{equation}\label{eq:GX}
G_X \hookrightarrow Aut(X),
\end{equation} where $Aut(X)$ denotes the algebraic group of abstract automorphisms of  $X$.
Since $X$, as an abstract variety, is isomorphic to
$\Pp(\mathcal E)$ over $\FF_1$,
$$\dim(Aut(X)) = \dim (Aut(\FF_1)) + \dim(Aut_{\FF_1} (\Pp(\mathcal E)),$$where $Aut_{\FF_1} (\Pp(\mathcal E)))$ denotes the automorphisms
of $\Pp(\mathcal E)$ fixing the base (cf. e.g. \cite{ma}).

From the fact that $Aut(\FF_1)$ is an algebraic group, in particular smooth, it follows that$$\dim(Aut(\FF_1)) = h^0(\FF_1, T_{\FF_1}) = 6$$(cf. \cite[Lemma 10, p. 106]{ma}), where $T_{\FF_1}$ denotes the tangent bundle of $\FF_1$.

On the other hand, $\dim(Aut_{\FF_1} (\Pp(\mathcal E)) = h^0(\mathcal E \otimes \mathcal E^{\vee}) - 1$, since $Aut_{\FF_1} (\Pp(\mathcal E))$ are given by endomorphisms of the projective bundle.

To sum up, $$\dim(Aut(X)) = h^0(\mathcal E \otimes \mathcal E^{\vee}) + 5.$$Since, from \eqref{eq:GX}, we have
$\dim(G_X) \leq \dim(Aut(X))$, then from \eqref{eq:dimY} we deduce
\begin{equation}\label{eq:dimY2}
\dim(\mathcal Y) \geq  \tau + n(n+2) - h^0(\mathcal E \otimes \mathcal E^{\vee}) - 5.
\end{equation}
According to the results in \S\,\ref{S:vbF1}, we have to distiguish two cases.
\vskip 3 pt
\noindent

$\bullet$ for $ b \leq k <  \frac{3b-3}{2}$, $\tau  = 0$ from  \eqref{eq:dimExt1b} and $h^0(\mathcal E \otimes \mathcal E^{\vee}) = 6b-4k-5$

from Lemma \ref{lem:autE}. Therefore from \eqref{eq:dimY2} $$\dim(\mathcal Y) \geq n(n+2) - 6b + 4 k.$$

\noindent

$\bullet$ for  $\frac{3b-3}{2} \leq k \leq  4b-8$, $\tau  = 4k-6b+6 $ from  \eqref{eq:dimExt1b} and $h^0(\mathcal E \otimes \mathcal E^{\vee}) = 1$

from Lemma \ref{lem:autE}. Once again, from \eqref{eq:dimY2}, $$\dim(\mathcal Y) \geq n (n+2) + 4k - 6 b + 6 - 6 =
n(n+2) - 6b + 4 k.$$
In any case, from \eqref{eq:expdim}, we have
$$n(n+2) - 6b + 4 k \leq \dim(\mathcal Y) \leq \dim(\mathcal X) = \dim(T_{[X]}(\mathcal X)) = n (n+1) + 3k - 2b -2.$$One can conclude
by observing that  $$n(n+2) - 6b + 4 k = n (n+1) + 3k - 2b -2.$$Indeed, this is equivalent to $n = 4b-k-2$, which is \eqref{eq:nd}.
\end{proof}
\begin{cor} $\dim(Aut(X)) = \dim (G_X)$. In other words, up to a possible extension via a finite group, all the abstract automorphisms of the projective scroll $X$ are induced by projective transformations.
\end{cor}

\begin{proof} It directly follows from the previous dimension computations and from
the orbit of $X$ via projective transformations.
\end{proof}

\section{Examples}\label{Examples}
\subsection{Scrolls over  $\FF_1$}  Going through the classification of manifolds of dimension $n\ge 3$  and degree  $10,11$, (\cite[Prop. 7.1]{fa-li10},   \cite[Prop. 4.2.3]{be-bi}), one finds  $3$-dimensional scrolls  $\xel =$ $ \scrollcal{E}$ over
$\FF_1$, of sectional genus $5$, degree $d=10, 11$, $c_1(\mathcal E) \num 3 C_0 + 5 f ,\; \; c_2(\mathcal E) = k=11, 10$, respectively,  which  are embedded in $\Pin{7},$  $\Pin{8},$ respectively. The existence of such 3-folds has been established in   \cite[Corollary 7.1]{al-be}.

The vector bundle  ${\mathcal E}$ satisfies the Assumptions \ref{ass:AB2}, hence by Proposition \ref{Hilbertscheme of scrollover F1} there exists an irreducible component $\mathcal X \subseteq \mathcal H$, which is generically smooth and of dimension
$\dim(\mathcal X) =  n(n+1) + 3k - 2b - 2$ and thus in our cases $\dim(\mathcal X) =  77, 90,$ respectively. Moreover, by  Theorem \ref{thm:genpo} the general point of $\mathcal X$ parametrizes a scroll $X$ as in the given examples.
The following Corollary summarizes the above discussion.
\begin{cor} Let $\xel = \scrollcal{E}$ be a 3-dimensional scroll over
$\FF_1.$ Let $X$ be  embedded by $|L|$ in $\Pin{n},$ with degree
$d$ and sectional genus $g$
as in the table below.
Then the  Hilbert scheme of $X \subset \Pin{n}$ has an irreducible
component, $\mathcal X$,  which is smooth at the point
representing $X$ and of dimension as in the rightmost column of the
table. Moreover the general point of $\mathcal X$ parametrizes a scroll $X$ as in the given examples.
\vskip .2in
\begin{tabular}{|c|c|c|c|c|c|c|}
\hline
$d$& $g$& $n$&$c_1(\mathcal{E})$&$c_2(\mathcal{E})$&  Reference &
$\dim{\mathcal{X}}$\\
\hline
\hline
$10$&$5$&$7$&$3 C_0 + 5 f$&$11$&\cite{fa-li10} Prop. 7.1&$77$\\
\hline
$11$&$5$&$8$&$3 C_0 + 5 f$&$10$&\cite{be-bi} Prop. 4.2.3
&$90$\\
\hline
\end{tabular}
\end{cor}

\subsection{Scrolls over $\FF_0$}
The Hilbert scheme of special  $3$-folds in $\Pin{n}$, $n\ge 6$, were studied by the first two authors, \cite{be-fa}. In particular, $3$-folds over smooth quadric surfaces $\FF_0=\Pp(\Oc_{\Pp^1}\oplus \Oc_{\Pp^1})$ were part of the cited work.  All such $3$-folds, which are known to exist, have sectional genus $4$, degree $8 \le  d \le  11$,  $c_1(\mathcal E) \num 3 C_0 + 3 f ,\; \; c_2(\mathcal E) = k=10, 9, 8, 7$, respectively, and are embedded in $\Pin{n},$  with $7\le n \le 10$, respectively.  They are shown to correspond to smooth points
of an irreducible component of their Hilbert scheme, whose dimension is computed:  $\dim \mathcal X=(20-k)(n-3)-3n+49$, (cf. \cite[Prop. 3.2]{be-fa} ). In this case the vector bundle ${\mathcal E}$ fits in an exact sequence
\begin{equation}\label{eq:al-beb1}
0 \to A \to \mathcal E \to B \to 0,
\end{equation}where $A$ and $B$ are such that
\begin{equation}\label{eq:al-be33}
A \num 2 C_0 + (6-k) f \;\; {\rm and} \;\; B \num C_0 + (k - 3) f.
\end{equation}
It is not difficult to see that
Theorem \ref {thm:genpo}  holds also if the base of the scroll is $\FF_0$.  In fact in these cases one can easily  see that $\dim \Pp(Ext^1(B,A))=4k-21.$
Twisting sequence (\ref{eq:al-beb1}) by $\mathcal E^{\vee}$ along with  some easy calculations gives that $h^{0}(\mathcal E\otimes \mathcal E^{\vee})=1$.  It is also known that $\dim(Aut(\FF_0)) = h^0(\FF_0, T_{\FF_0}) = 6$ (cf. \cite[Lemma 10, p. 106]{ma}).

 Letting  $\mathcal Y \subseteq \mathcal X$ be the locus filled-up by scrolls $X$ as in the above examples, one gets that $$\dim (\mathcal Y)\ge n(n+2)-27+4 k.$$
 Hence  we have
 $$n(n+2)-27+4 k \leq \dim(\mathcal Y) \leq \dim(\mathcal X) = \dim(T_{[X]}(\mathcal X)) = (20-k)(n -3) + 49-3n.$$One can conclude
by observing that  $$n(n+2)-27+4 k  = (20-k)(n -3) + 49-3n.$$Indeed, this is equivalent to $n +k= 16$, which holds in all the above  mentioned examples.


\end{document}